\documentclass[11pt]{amsart}
\usepackage{amsmath}
\usepackage{amsfonts}
\usepackage{amssymb}

\newtheorem{teo}{Theorem}[section]
\newtheorem{lem}[teo]{Lemma}
\newtheorem{cor}[teo]{Corollary}
\newtheorem{prop}[teo]{Proposition}

\newtheorem{defi}[teo]{Definition}
\newtheorem{oss}[teo]{Remark}

\newcommand{\dive}{\mathrm{div}}

\newcommand{\ud}{\mathrm{d}}

\renewcommand{\phi}{\varphi}
\newcommand{\R}{{\mathbb{R}}}

\newcommand{\bbR}{{\mathbb{R}}}

\def \R {{\mathbb {R}}}

\newcommand{\bbS}{\mathbb{S}}

\newcommand{\average}{{\mathchoice {\kern1ex\vcenter{\hrule height.4pt
width 6pt
depth0pt} \kern-9.7pt} {\kern1ex\vcenter{\hrule height.4pt width 4.3pt
depth0pt}
\kern-7pt} {} {} }}

\title[Symmetry results for fibered systems]{Symmetry results for stable and monotone solutions to fibered systems of PDEs}
\author{Serena Dipierro, Andrea Pinamonti}
\address{Departamento de Ingenier\'ia Matem\'atica, Universidad de Chile, Casilla 170 Correo 3, Santiago, Chile}
\email{sdipierro@dim.uchile.cl}
\address{Dipartimento di Matematica, Universit\`a di Padova, Via Trieste 63, Padova, Italia}
\email{pinamonti@science.unitn.it}

\begin{document}
\vskip .2truecm

\keywords{Elliptic systems, fibered media, monotone solutions, stable solutions, phase separation, Poincar\'{e}-type inequality.  \vspace{1mm}}

\thanks{The first author has been supported by FIRB ``Project Analysis and Beyond'' and by CAPDE-Anillo ACT-125. The second author has been supported by MIUR, Italy, GNAMPA of INDAM and by Fondazione CaRiPaRo Project ``Nonlinear Partial Differential Equations: models, analysis, and control-theoretic problems``}

\begin{abstract}
We study the symmetry properties for solutions of elliptic systems of the type
\begin{eqnarray*}
\left\{ 
\begin{array}{ll} 
 -\dive(a_1(x,|\nabla u^1|(X))\nabla u^1(X))=F_{1}(x, u^1(X),\ldots, u^n(X)),\\ 
\vdots\\
-\dive(a_n(x,|\nabla u^n|(X))\nabla u^n(X))=F_{n}(x, u^1(X),\ldots, u^n(X)),  
 \end{array} 
\right.
\end{eqnarray*} 
where~$x\in \R^m$ with $1\leq m< N$, $X=(x,y)\in \R^m\times \R^{N-m}$, 
and $F_{1},\ldots, F_{n}$ are the derivatives with respect to $\xi^1,\ldots, \xi^n$ 
of some $F=F(x,\xi^1,\ldots, \xi^n)$ such that for any $i=1,\ldots, n$ and any fixed $(x,\xi^1,\ldots, \xi^{i-1},\xi^{i+1},\ldots, \xi^n)\in \R^m\times \R^{n-1}$ 
the map $\xi^i\to F(x,\xi^1,\ldots,\xi^i,\ldots, \xi^n)$ belongs to $C^2(\R)$. 
We obtain a Poincar\'e-type formula for the solutions of the system and we use it to prove a symmetry result both for stable and for monotone solutions.
\end{abstract}

\maketitle

\section{Introduction}
In this paper we deal with symmetry results for solutions to the following system of partial differential equations defined in an open subset $\Omega$ of $\R^N$
\begin{equation}
\left\{
\begin{aligned}
\label{sistema}
-\dive&(a_1(x,|\nabla u^1|(X))\nabla u^1(X))=F_{1}(x, u^1(X),\ldots, u^n(X)),\\
&\vdots\\
-\dive&(a_n(x,|\nabla u^n|(X))\nabla u^n(X))=F_{n}(x, u^1(X),\ldots, u^n(X)).
\end{aligned}
\right.
\end{equation}
Here $x\in \R^m$ with $1\leq m< N$, $X=(x,y)\in \R^m\times \R^{N-m}$, 
and $F_{1},\ldots, F_{n}$ are the derivatives with respect to $\xi^1,\ldots, \xi^n$ 
of some $F=F(x,\xi^1,\ldots, \xi^n)$ such that, for any $i=1,\ldots, n$ and any fixed $(x,\xi^1,\ldots, \xi^{i-1},\xi^{i+1},\ldots, \xi^n)\in \R^m\times \R^{n-1}$, 
the map $\xi^i\to F(x,\xi^1,\ldots,\xi^i,\ldots, \xi^n)$ belongs to $C^2(\R)$. 
We also assume that
\begin{itemize}
	\item $a_i\in L^{\infty}(\R^m\times [\alpha_{-},\alpha_{+}])$, for any $\alpha_{+}>\alpha_{-}>0$, $i=1,\ldots,n$;
	\item for any fixed $t\in (0,+\infty)$ and any $i=1,\ldots, n$, $\inf_{x\in \R^m} a_i(x,t)>0$;
	\item for any fixed $x\in \R^m$ and any $i=1,\ldots, n$, the map $t\to a_i(x,t)$ is $C^1$ on $(0,+\infty)$.
\end{itemize}
The physical motivation for \eqref{sistema} comes from ``fibered'', or ``stratified'' media: namely, the medium, say $\Omega\subset\bbR^N$, is
nonhomogeneous, but this nonhomogeneity only occurs in lower dimensional slices (here, the medium is supposed to be homogeneous with respect to $y\in\bbR^{N-m}$ and nonhomogeneous with respect to $x\in \bbR^m$).

Systems similar to \eqref{sistema} have been studied in~\cite{BLWZ}. 
Precisely, the authors considered the following system
\begin{eqnarray}\label{syst}
\left\{ 
\begin{array}{ll} 
    \Delta u   = uv^2,         \\
    \Delta v = vu^2, \\
    u, v>0,
 \end{array} 
\right. 
\end{eqnarray} 
which arises in phase separation for multiple states Bose-Einstein 
condensates. 
They proved that there exists a solution to~\eqref{syst} in~$\bbR$, 
which is nondegenerate and reflectionally symmetric, 
namely that there exists~$x_0\in\bbR$ such that~$u(x-x_0)=v(x-x_0)$. 
Moreover, they obtained a result that may be seen as the analogue of a famous 
conjecture of De Giorgi for problem~\eqref{syst} in dimension~$2$, 
that is they proved that monotone solutions of~\eqref{syst} in~$\R^2$ (see Definition \ref{mon} below) have one-dimensional symmetry 
under the additional growth condition
\begin{equation}\label{growth}
  u(x)+v(x) \leq C(1+|x|). 
\end{equation}
On the other hand, in~\cite{NTTV}, it has been proved that 
the linear growth is the lowest possible for solutions to~\eqref{syst}; 
in other words, if there exists~$\alpha\in (0,1)$ such that 
$$
   u(x)+v(x) \leq C(1+|x|)^{\alpha},
$$
then~$u=v\equiv 0$. 

In~\cite{BSWW} the authors replaced the monotonicity condition by 
the stability of the solutions (which is a weaker assumption), 
showing that the above mentioned one-dimensional symmetry still holds in~$\R^2$. 
Moreover, they proved that there exist solutions to~\eqref{syst} 
which do not satisfy the growth condition~\eqref{growth}, 
constructing solutions with polynomial growth.  

We mention the paper~\cite{Wa}, 
where the author showed that, for any~$n\geq 2$, 
a solution to~\eqref{syst} which is a local minimizer and satisfies the growth condition~\eqref{growth} 
has one-dimensional symmetry. 

In~\cite{FG} it is proved that the symmetry result stated in~\cite{BLWZ} holds
also for a more general class of nonlinearities. 

Finally, in~\cite{Di}, the author considered a class of quasilinear (possibly degenerate) elliptic systems in~$\R^n$ and proved that, under suitable assumptions, the solutions have one-dimensional symmetry, 
showing that the results obtained in~\cite{BLWZ, BSWW, FG} hold in a more general setting. We also refer the reader to \cite{DP}, where symmetry results for systems driven by non local operators are studied.

Results similar to the ones described above are also well-understood in the case of one equation. In particular, in low dimensions, De Giorgi conjecture on the flatness 
of level sets of standard phase transition (\cite{degiorgi}) has been proved, 
see~\cite{AAC, AC, BCN, GG, GG1}. 
Later, Savin in~\cite{Sav} showed that the conjecture is true up to dimension~$8$ 
under an additional hypothesis on the behaviour of the solution at infinity. 
Finally, in dimension~$n\geq9$ Del Pino, Kowalczyk and Wei constructed a solution 
to the Allen-Cahn equation which is monotone in one direction but not one-dimensional, 
see~\cite{DKW}. 

It is also worth noticing that an analogous of the De Giorgi conjecture has been studied for more general operators. In particular, we mention~\cite{FSV}, 
where quasilinear operators of p-Laplacian and curvature type are considered, 
and~\cite{CSM, SV}, where the authors proved a similar De Giorgi-type result for an equation involving the fractional Laplacian in dimension $n=2$; 
see also~\cite{CaCi,CSir, CSir2} for further extensions.

\bigskip

First of all, we give the following definition:
\begin{defi}\label{weak}
An $n-$uple $(u^1,\ldots, u^n)$ is said to be a weak solution of (\ref{sistema}) if, 
for any $\psi=(\psi^1,\ldots, \psi^n)\in C^{\infty}_{c}(\Omega,\R^n)$,
\begin{align}\label{weak1}
		\int_{\Omega}\left\langle a_i(x,|\nabla u^i|(X))\nabla u^i(X),\nabla\psi^i(X) \right\rangle\ud X=\int_{\Omega} F_{i}(x,u^1,\ldots, u^n)\, \psi^i(X)\ud X, \quad i=1,\ldots, n.
\end{align}
\end{defi}
In order to state our results we start pointing out our assumptions. In particular, from now on we will always assume that every weak solution $(u^1,\ldots, u^n)$ of (\ref{sistema}) is such that\footnote{At the end of this paper we present some explicit cases in which these assumptions are fulfilled.}
\begin{equation}\begin{split}\label{regassump} 
& u^i\in C^1(\Omega)\cap C^2(\Omega\cap \{\nabla u^i\neq 0\})\cap L^{\infty}(\Omega) \\
& {\mbox{and\ }} \nabla u^i\in L^{\infty}(\Omega,\R^N)\cap W^{1,2}_{loc}(\Omega,\R^N), \quad i=1,\ldots,n.
\end{split}\end{equation}
Moreover, we will also assume that, for every $k=1,\ldots, n$, the map 
$$\mathcal{A}^k:\R^m\times (\R^N\setminus \{0\})\longrightarrow Mat(N\times N)$$
defined by\footnote{Here, as usual, $Mat(N\times N)$ denotes the vector space of real $N\times N$ matrices.}
\begin{align*}
	&\mathcal{A}_{ij}^k=\mathcal{A}_{ij}^k(x,\eta):=a_k(x,|\eta|)\delta_{ij}+\frac{\partial a_k}{\partial t}(x,|\eta|)\frac{\eta_i\eta_j}{|\eta|}, \qquad 1\leq i,j\leq N
\end{align*}
is such that\footnote{Let us observe that condition (\ref{Alim}) is implied if, for example, $\frac{\partial a_k}{\partial t}\in L^{\infty}(\bbR^m\times [\alpha-,\alpha+])$,  for all $\alpha_{+}>\alpha_{-}>0$.}
\begin{equation}\begin{split}\label{Alim}
&(x,y)=X\to \mathcal{A}^i(x,\nabla u^i(X))
{\mbox{\ belongs to $L^{\infty}(\{\nabla u^i\neq 0\}\cap B_R)$}} \\
&{\mbox{for any $R>0$ and any $i=1,\ldots, n$.}}
\end{split}\end{equation}
\vspace{5pt}

The following definition was introduced in \cite{FG}.
\begin{defi}\label{mon}
A solution $(u^1,\ldots, u^n)$ of (\ref{sistema}) is said to be $F-$monotone if 
\begin{enumerate}
\item[i)] for every $i\in\{1,\ldots, n\}$, $\partial_{y_{N-m}}u^i\neq 0$ in $\Omega$,
\item[ii)] for $i<j$, we have $F_{ij}\, \partial_{y_{N-m}}\, u^i\partial_{y_{N-m}}u^j\geq 0$.
\end{enumerate}
\end{defi}
As it is customary in this setting, we recall the notion of stability:
\begin{defi}\label{defstab}
A solution $(u^1,\ldots, u^n)$ of (\ref{sistema}) is said to be stable if 
\begin{equation}\begin{split}\label{stab}
& \sum_{i=1}^n\int_{\Omega}\left\langle \mathcal{A}^i(x,\nabla u^i(X))\nabla \psi^i(X), \nabla \psi^i(X)\right\rangle\ud X \\
&\qquad\qquad -\sum_{i,j=1}^n\int_{\Omega} F_{ij}(x,u^1,\ldots, u^n)\, \psi^i(X)\, \psi^j(X)\ud X\geq 0,
\end{split}\end{equation}
for any~$\psi=(\psi^1,\ldots, \psi^n)\in C^{\infty}_c(\Omega,\R^n)$. 
\end{defi}
Let us note that~\eqref{sistema} represents the Euler-Lagrange system associated to a suitable energy functional $I$ (see the appendix). 
In particular, the notion of stability given in Definition~\ref{defstab} states that $I$ has positive (formal) second variation (we refer to \cite{AAC,AC,FSV} for more details). 

According to \cite{CVal,FSV,SZ1}, for every fixed $x\in\R^m$ and $c\in\R$, we define
\[H_u=H_{u,x,c}:=\{y\in \R^{N-m}\ |\ u(x,y)=c\} \]
and 
\[L_u=L_{u,x,c}:=\{y\in H_u\ |\ \nabla_y u(x,y)\neq 0\}. \]
We also define
\begin{align*}
	&\mathcal{R}_u:=\{(x,y)\in\Omega\ |\ \nabla_y u(x,y)\neq 0\},
	&\mathcal{S}_u:=\sum_{i=1}^m\sum_{j=1}^{N-m}(u_{x_iy_j})^2-|\nabla_x|\nabla_y u||^2,\\
	&\mathcal{T}_u:=\sum_{j=1}^{N-m}\left\langle\nabla u,\nabla u_{y_j}\right\rangle^2-\left\langle \nabla u,\nabla|\nabla_y u|\right\rangle^2,
  &\mathcal{U}_u:=\sum_{j=1}^{N-m}|\nabla u_{y_j}|^2-|\nabla|\nabla_y u||^2.
\end{align*}
We recall that the tangential gradient along $L_u$, $\nabla_L$, is defined for every $\bar y\in L_u$ and any $G:\R^{N-m}\longrightarrow \R$ smooth in the vicinity of $\bar y$ as
	\[\nabla_L G(\bar y):=\nabla_y G(\bar y)-\left\langle\nabla_y G(\bar y), \frac{\nabla_y u(x,\bar y)}{|\nabla_y u(x,\bar y)|}\right\rangle\frac{\nabla_y u(x,\bar y)}{|\nabla_y u(x,\bar y)|},
\]
and since $L_u$ is a smooth $(N-m-1)-$manifold we define, for every $y\in L_u$, 
the length of the second fundamental form by
	\[\mathcal{K}_u(x,y):=\sqrt{\sum_{j=1}^{N-m-1}k_j^2(x,y)},
\]
where $k_{1,u}(x,y),\ldots, k_{N-m-1,u}(x,y)$ are the principal curvatures of $L_u$.

\vspace{5pt}

We are now in position to state our main results. 
We establish first a geometric inequality, which involves the tangential gradients 
and the curvatures of the level sets of the solution. 
This inequality holds in every open set~$\Omega\subset\bbR^N$.
\begin{teo}\label{ineq} 
Let $(u^1,\ldots, u^n)$ be a weak stable solution of (\ref{sistema}) satisfying (\ref{regassump}).

Then, for each $\psi=(\psi^1,\ldots,\psi^n)\in C^{\infty}_c(\Omega,\R^n)$, we have 
\begin{align}\label{stat1}
&\sum_{k=1}^n \int_{\mathcal{R}_{u^k}}\left(\sum_{j=1}^{N-m} \left\langle \mathcal{A}^k\nabla u^k_{y_j}, \nabla u^k_{y_j}\right\rangle-\left\langle \mathcal{A}^k\nabla|\nabla_y u^k|,\nabla|\nabla_y u^k|\right\rangle\right)(\psi^k)^2 \\
\nonumber
&\quad -\sum_{k,j=1, j\neq k}^n\int_{\Omega} F_{kj}\left((\psi^k)^2\left\langle \nabla_y u^j, \nabla_y u^k\right\rangle-\psi^j\psi^k|\nabla_y u^k||\nabla_y u^j|\right)\\
	\nonumber
\leq & \int_{\Omega}\sum_{k=1}^n\left\langle \mathcal{A}^k\nabla\psi^k,\nabla\psi^k\right\rangle|\nabla_y u^k|^2.
\end{align}
Moreover,
\begin{align*}
& \sum_{k=1}^n\int_{R_{u^k}}\left[a_k(x,|\nabla u^k|)(\mathcal{S}_{u^k}+\mathcal{K}_{u^k}^2|\nabla_y u^k|^2+|\nabla_L|\nabla_y u^k||^2)+\frac{\frac{\partial{a}_k}{\partial t}(x,|\nabla u^k|)}{|\nabla u^k|}\mathcal{T}_{u^k}\right](\psi^k)^2 \\
	\nonumber
\leq & \int_{\Omega}\sum_{k=1}^n\left\langle \mathcal{A}^k\nabla\psi^k,\nabla\psi^k\right\rangle|\nabla_y u^k|^2 \\ 
\nonumber 
&\quad -\sum_{k,j=1, j\neq k}^n\int_{\Omega} F_{kj}\left((\psi^k)^2|\nabla_y u^k||\nabla_y u^j|-\psi^j\psi^k\left\langle \nabla_y u^i, \nabla_y u^k\right\rangle\right).
\end{align*}
\end{teo}

Next, we state our symmetry results both for stable and for monotone solutions to~\eqref{sistema}. 
In the proof of the subsequent theorems, 
we will use the geometric inequality in~\eqref{stat1} with~$\Omega=\bbR^N$.
\begin{teo}\label{teoprinc}
Let $(u^1,\ldots, u^n)$ be a weak stable solution of (\ref{sistema}) in the whole $\R^N$ satisfying~(\ref{regassump}). Let us also assume that there exist non-zero functions~$\theta^i\in C^1(\R^N)$, $i=1,\ldots,n$, which do not change sign, such that for all $i,j$ with $1\leq i<j\leq n$ it holds
\begin{align}\label{assumdermiste}
	F_{ij}(x,u^1(X),\ldots, u^n(X))\, \theta^i(X)\, \theta^j(X)\geq 0, \qquad \forall X\in\R^N.
\end{align}
Moreover, we assume that, for any $i=1,\ldots, n$, $\mathcal{A}^i(x,\nabla u^i(X))$ satisfies (\ref{Alim}), it is positive definite at almost any $X\in\R^N$ and there exist $C_1,\ldots, C_n\geq 1$ such that the largest eigenvalue $\overline{\mathcal{A}}^i(X)$ of $\mathcal{A}^i(x,\nabla u^i(X))$ satisfies 
\begin{align}\label{assump1}
	\int_{B_R}\overline{\mathcal{A}}^i(X)|\nabla u^i(X)|^2\ud X\leq C_iR^2
\end{align}
for any $R\geq \max\{C_1,\ldots, C_n\}$.  

Then, for each $i=1,\ldots, n$, there exist $\bar u^i:\R^m\times\R\longrightarrow \R$ and $\omega_i:\R^m\longrightarrow \bbS^{N-m-1}$ such that 
	\[u^i(X)=u^i(x,y)=\bar u^i(x, \left\langle\omega_i(x), y\right\rangle)
\]
for any $(x,y)\in\R^m\times \R^{N-m}$. Moreover, each $\omega_{i}(x)$ is constant in any connected component of $\{\nabla_y {u^i}\neq 0\}$.
\end{teo}

\begin{teo}\label{teoprinc2}
Let $(u^1,\ldots, u^n)$ be a weak $F-$monotone solution of (\ref{sistema}) in the whole $\R^N$ satisfying (\ref{regassump}). We assume that, for any $i=1,\ldots, n$, $\mathcal{A}^i(x,\nabla u^i(X))$ satisfies (\ref{Alim}), it is positive definite at almost any $X\in\R^N$ and $\overline{\mathcal{A}}^i(X)$ satisfies (\ref{assump1}). 

Then, there exist $\bar u^i:\R^m\times\R\longrightarrow \R$ and $\omega_i\in \bbS^{N-m-1}$ such that 
	\[u^i(X)=u^i(x,y)=\bar u^i(x, \left\langle \omega_i, y\right\rangle)
\]
for any $(x,y)\in\R^m\times \R^{N-m}$. 

If, in addition, there exists $U\subseteq\bbR^m\times \bbR^n$ open such that, for every $j,k=1,\ldots, n$, $F_{jk}>0$ (or $F_{jk}<0$) in $U$, then $\omega_j=\omega_k=\omega$.
\end{teo}

Some comments are now in order. We start recalling that a result similar 
to the one in Theorem \ref{ineq} has been proved in \cite{CVal,CValp,Far,FSV,SV1,SZ1,SZ2} in the case of one equation and in \cite{FG,dipierro} for particular systems of equations. Precisely, a geometric inequality like the one in~\eqref{stat1} 
has been obtained in \cite{FG} in the case
$$a_i(x,t)=1, \qquad F_{i}(x,u^1,\ldots, u^n)=F_{i}(u^1,\ldots, u^n), $$
and in \cite{dipierro}, for $n=2$, in the case 
$$ a_i(x,t)=a_i(t), \qquad F_{i}(x,u^1,u^2)=F_{i}(u^1,u^2).$$

We also mention the papers \cite{fsv1,fsv2,PV}, 
where an inequality similar to the one obtained in Theorem \ref{ineq} has been 
established for solutions to semilinear problems in Riemannian and sub-Riemannian manifolds. 

We remark that Theorems \ref{teoprinc} and \ref{teoprinc2} generalize to fibered media the results contained in \cite{FG,dipierro} and allow us to take into account more general systems than the ones considered in \cite{FG,dipierro} (see also the appendix for some explicit examples). 

\bigskip

The paper is organized as follows. In the next section we prove the geometric 
inequality in Theorem~\ref{ineq}. 
Sections $3$ and $4$ are devoted to the proof of Theorems~\ref{teoprinc} and~\ref{teoprinc2}. 
Finally, there is an Appendix, which contains some comments on 
the assumptions made in our theorems.

\section{A geometric inequality: proof of Theorem \ref{ineq}}
Aim of this section is to prove Theorem \ref{ineq}. 
We recall first the following lemma, which has been proved in \cite{CVal, SZ1,SZ2}.
\begin{lem} \label{ugugeom} 
For any $u\in C^{2}(\Omega)$, the following equalities hold
\begin{align*}
& \sum_{j=1}^{N-m}\left\langle \mathcal{A}\nabla u_{y_j},\nabla u_{y_j}\right\rangle-\left\langle \mathcal{A}\nabla|\nabla_y u|,\nabla|\nabla_y u|\right\rangle=a(x,|\nabla u|)\, \mathcal{U}_u+\frac{\frac{\partial a}{\partial t}(x,|\nabla u|)}{|\nabla u|}\mathcal{T}_u \quad \mbox{on}\ \ \mathcal{R}_u,\\
& \mathcal{U}_u-\mathcal{S}_u=\sum_{i,j=1}^{N-m}(u_{y_iy_j})^2-|\nabla_y|\nabla_y u||^2=\mathcal{K}_u^2|\nabla_y u|^2+|\nabla_L|\nabla_y u||^2\quad \mbox{on}\ \ \mathcal{R}_u.
\end{align*}
Moreover, $\mathcal{S}_u,\mathcal{T}_u\geq 0$ on $\mathcal{R}_u$. 
\end{lem}

In the sequel, we will need also the following result.
\begin{prop}\label{dersistema}
Let $(u^1,\ldots, u^n)$ be a weak solution of (\ref{sistema}) satisfying (\ref{regassump}). 
Suppose that, for each $i=1,\ldots, n$, $\mathcal{A}^i$ verifies (\ref{Alim}). 

Then, for every $j=1,\ldots, N-m$, the family $(u^1_{y_j},\ldots, u^n_{y_j})$ is a weak solution of the following system
\begin{equation*}
\left\{
\begin{aligned}
&-\dive(\mathcal{A}^1(x,\nabla u^1(X))\nabla u^1_{y_j}(X))=\sum_{i=1}^nF_{1i}(x, u^1(X),\ldots, u^n(X))\, u^i_{y_j}(X),\\
&\vdots\\
&-\dive(\mathcal{A}^n(x,\nabla u^n(X))\nabla u^n_{y_j}(X))=\sum_{i=1}^nF_{ni}(x, u^1(X),\ldots, u^n(X))\, u^i_{y_j}(X).
\end{aligned}
\right.
\end{equation*}
\end{prop}
\begin{proof}
We need to prove that, for every 
$\psi=(\psi^1,\ldots, \psi^n)\in C^{\infty}_c(\Omega,\bbR^n)$, 
the following equalities hold (we drop for short the dependence of~$\mathcal{A}^i$):
\begin{equation}
\left\{
\begin{aligned}
\label{sistema4}
&\int_{\Omega}\left\langle\mathcal{A}^1 \nabla u^1_{y_j}(X),\nabla \psi^1(X)\right\rangle=\sum_{i=1}^n\int_{\Omega}F_{1i}(x, u^1(X),\ldots, u^n(X))\, u^i_{y_j}(X)\, \psi^1(X),\\
&\vdots\\
&\int_{\Omega}\left\langle\mathcal{A}^n \nabla u^n_{y_j}(X)\nabla \psi^n(X)\right\rangle=\sum_{i=1}^n\int_{\Omega}F_{ni}(x, u^1(X),\ldots, u^n(X))\, u^i_{y_j}(X)\, \psi^n(X).
\end{aligned}
\right.
\end{equation}
To this end, we fix $i\in \{1,\ldots, N-m\}$, and we use (\ref{weak1}) with $(\psi^1_{y_i},\ldots, \psi^n_{y_i})$ as test functions. Hence (dropping for short the dependence of $F_j$ and $a_j$), we have
\begin{align}\label{A}
\int_{\Omega} F_{j}\, \psi^j_{y_i}=\int_{\Omega} \left\langle a_j\nabla u^j,\nabla \psi^j_{y_i}\right\rangle=-\int_{\Omega}\left\langle(a_j\nabla u^j)_{y_i},\nabla\psi^j \right\rangle=-\int_{\Omega}\left\langle \mathcal{A}^j\nabla u_{y_i},\nabla\psi^j\right\rangle.
\end{align}
Moreover, 
\begin{align}\label{F}
\int_{\Omega} F_{j}\, \psi^j_{y_i}=-\int_{\Omega} (F_j)_{y_i}\, \psi^j=-\sum_{k=1}^n\int_{\Omega} F_{jk}\, u^k_{y_i}\, \psi^j,
\end{align}
and putting together (\ref{A}) and (\ref{F}) we get the thesis.
\end{proof}

\begin{oss}\label{density}
By an easy density argument (see \cite{CValp} for the details), we have that 
(\ref{sistema4}) holds for any $\psi=(\psi^1, \ldots, \psi^n)\in W^{1,2}_0(\Omega, \bbR^n)$.
\end{oss}

\begin{proof}[\textbf{Proof of Theorem \ref{ineq}:}]
Let us fix $1\leq j\leq N-m$ and $\psi=(\psi^1,\ldots, \psi^n)\in C^{\infty}_c(\Omega,\R^n)$. Then, by (\ref{regassump}), we have that 
$\varphi^i:=u^i_{y_j}(\psi^i)^2\in W^{1,2}_0(\Omega)$. 
Hence, by Remark \ref{density}, we can use $\varphi=(\varphi^1, \ldots, \varphi^n)$ 
as a test function in (\ref{sistema4}). 
It follows that (dropping for short the dependence of $\mathcal{A}^i$ and $F_{ij}$), for any $k=1,\ldots n$,
\begin{align*}
\int_{\Omega}\left\langle \mathcal{A}^k\nabla u^k_{y_j},\nabla(u^k_{y_j}(\psi^k)^2)\right\rangle = \sum_{i=1}^n \int_{\Omega} F_{ki}\, u^i_{y_j}\, u^k_{y_j}\, (\psi^k)^2.
\end{align*}
Summing over $j$, we have
\begin{align*}
\sum_{j=1}^{N-m}\int_{\Omega} \left\langle \mathcal{A}^k \nabla u^k_{y_j},\nabla(u^k_{y_j}(\psi^k)^2)\right\rangle=\sum_{i=1}^n \int_{\Omega} F_{ki}\, (\psi^k)^2\left\langle \nabla_y u^i, \nabla_y u^k\right\rangle,
\end{align*}
which implies 
\begin{equation}\begin{split}\label{MVal1} 
 \int_{\Omega} F_{kk}\, (\psi^k)^2 |\nabla_y u^k|^2 = & 
\sum_{j=1}^{N-m} \int_{\Omega} \left\langle \mathcal{A}^k \nabla u^k_{y_j},\nabla(u^k_{y_j}(\psi^k)^2)\right\rangle \\ 
&\quad - \sum_{i=1,i\neq k}^n \int_{\Omega} F_{ki}\, (\psi^k)^2\left\langle \nabla_y u^i, \nabla_y u^k\right\rangle.
\end{split}\end{equation}
Using Stampacchia's Theorem (see, for istance, \cite[Theorem $6.19$]{Stamp}), we get
\begin{align}\label{st}
		\nabla |\nabla_y u^k|=0=\nabla u^k_{y_j}\quad \mbox{for a.e.}\ x\in\bbR^m,\ \mbox{and a.e.}\ y\in\bbR^{N-m}\ s.t.\ \nabla_y u(x,y)=0.
\end{align}
Hence, summing over $k=1,\ldots, n$ in \eqref{MVal1}, we obtain
\begin{equation}\begin{split}\label{nuova}
\sum_{k=1}^n \int_{\Omega} F_{kk}\, (\psi^k)^2 |\nabla_y u^k|^2 = & 
\sum_{k=1}^n \int_{\mathcal R_{u^k}}\sum_{j=1}^{N-m} \left\langle \mathcal{A}^k \nabla u^k_{y_j},\nabla(u^k_{y_j}(\psi^k)^2)\right\rangle \\
&\quad - \sum_{k,i=1,i\neq k}^n \int_{\Omega} F_{ki}\, (\psi^k)^2\left\langle \nabla_y u^i, \nabla_y u^k\right\rangle.
\end{split}\end{equation}
 
Using $\varphi^k:=|\nabla_y u^k|\psi^k$, with $\psi=(\psi^1,\ldots, \psi^n)\in C^{\infty}_c(\Omega,\R^n)$, 
as a test function in~(\ref{stab}) (we point out that by the regularity assumptions on $u^i$ it follows that $\varphi^i\in W^{1,2}_0(\Omega)$, $i=1,\ldots, n$), and 
recalling~\eqref{st}, we obtain 
\begin{align*}
0\leq& \sum_{k=1}^n \int_{\mathcal{R}_{u^k}}\left\langle \mathcal{A}^k\nabla|\nabla_y u^k|,\nabla|\nabla_y u^k|\right\rangle(\psi^k)^2+\left\langle \mathcal{A}^k\nabla\psi^k,\nabla\psi^k\right\rangle|\nabla_y u^k|^2\\
\nonumber	
&\quad +2\left\langle \mathcal{A}^k\nabla|\nabla_y u^k|,\nabla \psi^k\right\rangle|\nabla_y u^k|\psi^k\\
\nonumber
&\quad -\sum_{k,j=1, j\neq k}^n\int_{\Omega} F_{kj}|\nabla_y u^k||\nabla_y u^j|\psi^j\psi^k-\sum_{k=1}^n\int_{\Omega} F_{kk}|\nabla_y u^k|^2(\psi^k)^2,
\end{align*}
which together with (\ref{nuova}) implies 
\begin{align}\label{ineqfond}
0 \leq & \sum_{k=1}^n \int_{\mathcal{R}_{u^k}}\left\langle \mathcal{A}^k\nabla|\nabla_y u^k|,\nabla|\nabla_y u^k|\right\rangle(\psi^k)^2+\left\langle \mathcal{A}^k\nabla\psi^k,\nabla\psi^k\right\rangle|\nabla_y u^k|^2\\
\nonumber
&\quad +\frac{1}{2}\left\langle \mathcal{A}^k\nabla|\nabla_y u^k|^2,\nabla (\psi^k)^2\right\rangle \\
\nonumber
&\quad +\sum_{k,j=1, j\neq k}^n\int_{\Omega} F_{kj}\left((\psi^k)^2\left\langle \nabla_y u^j, \nabla_y u^k\right\rangle-\psi^j\psi^k|\nabla_y u^k||\nabla_y u^j|\right)\\
\nonumber
&\quad -\sum_{k=1}^n\int_{\mathcal{R}_{u^k}}\sum_{j=1}^{N-m} (\psi^k)^2\left\langle \mathcal{A}^k\nabla u^k_{y_j}, \nabla u^k_{y_j}\right\rangle \\
\nonumber
&\quad -\sum_{k=1}^n\frac{1}{2}\int_{\mathcal{R}_{u^k}}\mathcal{A}^k\left\langle \nabla |\nabla_y u^k|^2, \nabla(\psi^k)^2\right\rangle.
\end{align}
Rewriting the inequality in~(\ref{ineqfond}) in a more compact form, we obtain
\begin{align*}
0 \leq & \sum_{k=1}^n \int_{\mathcal{R}_{u^k}}\Big[\Big(\left\langle \mathcal{A}^k\nabla|\nabla_y u^k|,\nabla|\nabla_y u^k|\right\rangle-\sum_{j=1}^{N-m} \left\langle \mathcal{A}^k\nabla u^k_{y_j}, \nabla u^k_{y_j}\right\rangle\Big)(\psi^k)^2 \\
\nonumber
&\qquad\qquad +\left\langle \mathcal{A}^k\nabla\psi^k,\nabla\psi^k\right\rangle|\nabla_y u^k|^2\Big]\\
\nonumber
&\quad +\sum_{k,j=1, j\neq k}^n\int_{\Omega} F_{kj}\Big((\psi^k)^2\left\langle \nabla_y u^j, \nabla_y u^k\right\rangle-\psi^j\psi^k|\nabla_y u^k||\nabla_y u^j|\Big),
\end{align*}
which is the first part of the statement. For the second part, we observe that, by Lemma \ref{ugugeom}, we get
\begin{align}\label{wep}
&\left\langle \mathcal{A}^k\nabla|\nabla_y u^k|,\nabla|\nabla_y u^k|\right\rangle-\sum_{j=1}^{N-m}\left\langle \mathcal{A}^k\nabla u^k_{y_j},\nabla u^k_{y_j}\right\rangle\\
	\nonumber
= & -a_k(x,|\nabla u^k|)(\mathcal{S}_{u^k}+\mathcal{K}_{u^k}^2|\nabla_y u^k|^2+|\nabla_L|\nabla_y u^k||^2)-\frac{\frac{\partial a_k}{\partial t}(x,|\nabla u^k|)}{|\nabla u^k|}\mathcal{T}_{u^k}\quad \mbox{on}\ \mathcal{R}_{u^k}.
\end{align}
Therefore, plugging (\ref{wep}), into (\ref{ineqfond}) we get the thesis.
\end{proof}

\section{Stable solutions and proof of Theorem \ref{teoprinc}}
Recalling the definition of stable solutions to~\eqref{sistema} given in~\eqref{stab}, 
in this section we will prove Theorem \ref{teoprinc}. 

First, we recall the following lemma from \cite{CValp}.
\begin{lem}\label{stima}
Let $R>0$ and $h:B_R\subset\R^N\longrightarrow \R$ be a nonnegative measurable function. For any $\rho\in (0,R)$, let
\[\xi(\rho):=2\int_{B_{\rho}} h(X)\ud X. \]

Then,
\[\int_{B_R\setminus B_{\sqrt{R}}}\frac{h(X)}{|X|^2}\ud X\leq \int_{\sqrt{R}}^R t^{-3}\xi(t)\ud t+\frac{\xi(R)}{R^2}.
\]
\end{lem}

\begin{proof}[\textbf{Proof of Theorem \ref{teoprinc}}:] 
In order to prove Theorem \ref{teoprinc}, we use the geometric inequality in 
(\ref{stat1}), with~$\Omega=\bbR^N$. 
Since $\mathcal{A}^k$ is positive definite, inequality (\ref{stat1}) becomes
\begin{align}\label{ineq3}
& \sum_{k=1}^n \int_{\mathcal{R}_{u^k}}\left(\sum_{j=1}^{N-m} \left\langle \mathcal{A}^k\nabla u^k_{y_j}, \nabla u^k_{y_j}\right\rangle-\left\langle \mathcal{A}^k\nabla|\nabla_y u^k|,\nabla|\nabla_y u^k|\right\rangle\right)(\psi^k)^2 \\
\nonumber
&\quad -\sum_{k,j=1, j\neq k}^n\int_{\bbR^N} F_{kj}\Big((\psi^k)^2\left\langle \nabla_y u^j, \nabla_y u^k\right\rangle-\psi^j\psi^k|\nabla_y u^k||\nabla_y u^j|\Big)\\
	\nonumber
\leq & \sum_{k=1}^n\int_{\bbR^N} \overline{ \mathcal{A}}^k|\nabla\psi^k|^2|\nabla_y u^k|^2.
\end{align}

By Lemma $3.1$ in \cite{CVal}, we have 
\begin{align*}
& \sum_{j=1}^{N-m}\left\langle \mathcal{A}^k\nabla u^k_{y_j},\nabla u^k_{y_j}\right\rangle-\left\langle \mathcal{A}^k\nabla|\nabla_y u^k|,\nabla|\nabla_y u^k|\right\rangle\geq 0, \qquad k=1,\ldots, n.
\end{align*}
Moreover, by (\ref{assumdermiste}), we have that there exist non-zero functions $\theta^1,\ldots, \theta^n\in C^1(\R^N)$ with constant sign such that
\begin{align}\label{222}
	F_{ij}\, \theta^i\, \theta^j\geq 0, \qquad \forall i,j\in\{1,\ldots, n\}, i<j.
\end{align}

For any $R>1$, we define $\eta_R:\R^N\longrightarrow\R$ as
\begin{equation*}
\eta_R(X):=\left\{
\begin{aligned}
&1\ \ \ \ \ \ \ \ \ \ \ \ \ \ \ \ \ \ \ \ \ \ \ \ \ \ \ \ \ \mbox{if}\  X\in B_{\sqrt{R}},\\
&2\, \frac{\log R-\log |X|}{\log R}\ \ \ \ \ \mbox{if}\ X\in B_{R}\setminus B_{\sqrt{R}},\\
&0\ \ \ \ \ \ \ \ \ \ \ \ \ \ \ \ \ \ \ \ \ \ \ \ \ \ \ \ \ \mbox{if}\ X\in \R^N\setminus B_R,
\end{aligned}
\right.
\end{equation*}
and consider 
\begin{align}\label{deftest}
	\eta^i_R:=sgn(\theta^i)\eta_R,
\end{align}
where $sgn(x)$ is the Sign function. It follows that, for each $i=1,\ldots, n$, $\eta^i_R\in C^{\infty}_c(B_R)$, $0<|\eta^i_R(X)|<1$ for any $X\in\R^N$, and 
	\[|\nabla \eta^i_R(X)|\leq \frac{\chi_R(X)}{2|X|\log R},
\]
where
\begin{equation*}
\chi_R(X):=\left\{
\begin{aligned}
&1\ \ \ \ \ \ \ \  \mbox{if}\  X\in B_{R}\setminus B_{\sqrt{R}},\\
&0\ \ \ \ \ \ \ \  \mbox{otherwise}.
\end{aligned}
\right.
\end{equation*}
Moreover, from~\eqref{222}, we have
\begin{align*}
	F_{ij}\, sgn(\theta^i)\, sgn(\theta^j)\geq 0, \qquad \forall i,j\in\{1,\ldots, n\}, i<j.
\end{align*}

Using (\ref{deftest}) as a test function in (\ref{ineq3}), and observing that 
$$F_{kj}\, sgn(\theta^k)\, sgn(\theta^j)=sgn(F_{kj})F_{kj}, $$ 
we get
\begin{align}\label{qp}
& \sum_{k=1}^n\int_{\mathcal R_{u^k}}\left[\sum_{j=1}^{N-m}\left\langle \mathcal{A}^k\nabla u^k_{y_j},\nabla u^k_{y_j}\right\rangle-\left\langle \mathcal{A}^k\nabla|\nabla_y u^k|,\nabla|\nabla_y u^k|\right\rangle\right]\left(\eta^k_R\right)^2 \\
	\nonumber
&\quad -\sum_{k,j=1, j\neq k}^n \int_{\bbR^N}  \Big(sgn(F_{kj})\left\langle \nabla_y u^j, \nabla_y u^k\right\rangle-|\nabla_y u^k||\nabla_y u^j|\Big)F_{kj}\, sgn(F_{kj})\, \eta_R^2 \\
	\nonumber
\leq & \frac{1}{4\log^2{R}}\sum_{k=1}^n\int_{B_{R}\setminus B_{\sqrt{R}}}\frac{\overline{\mathcal{A}}^k|\nabla_y u^k|^2}{|X|^2}\\
	\nonumber
\leq & C\frac{1}{\log{R}}, 
\end{align}
where in the last inequality we have used the fact that $|\nabla_y u|^2\leq |\nabla u|^2$, Lemma \ref{stima} with $h(X):=\overline{\mathcal{A}}^k|\nabla_y u^k|^2$ and 
the assumption (\ref{assump1}). 

Sending $R\to +\infty$ in (\ref{qp}), we conclude that
\begin{equation}\label{oss3}
\sum_{j=1}^{N-m}\left\langle \mathcal{A}^k\nabla u^k_{y_j},\nabla u^k_{y_j}\right\rangle-\left\langle \mathcal{A}^k\nabla|\nabla_y u^k|,\nabla|\nabla_y u^k|\right\rangle= 0, \qquad \mbox{a.e. in}\ \mathcal{R}_{u^k},\ k=1,\ldots,n,\\
\end{equation}
and
\begin{equation}\label{oss5}
\Big(sgn(F_{kj})\left\langle \nabla_y u^j, \nabla_y u^k\right\rangle-|\nabla_y u^k||\nabla_y u^j|\Big)sgn(F_{kj})F_{kj}=0, \qquad \mbox{a.e. in}\ \bbR^N, 
\end{equation}
for any $k,j=1,\ldots, n$, with $j\neq k$.

By (\ref{oss3}) and Corollary $3.2$ in \cite{CVal}, we obtain that, 
for any level set $L$ of $u^k$ and any $X\in\mathcal R_{u^k}\cap L$, 
\begin{align*}
	\mathcal{K}_{u^k}=0=|\nabla_L|\nabla_y u^k||. 
\end{align*}
Therefore, using Lemma $2.11$ in \cite{FSV}, this implies that, 
for each $k=1,\ldots, n$, there exist 
$\omega_{k}:\R^m\longrightarrow\bbS^{N-m-1}$ and $\bar u^k:\R^m\times \R\longrightarrow\R$ such that
	\[u^k(x,y)=\bar u^k(x,\left\langle \omega_{k}(x), y\right\rangle).
\]

Moreover, by Lemma A.1 in \cite{CValp} we have that each $\omega_{k}(x)$ is constant in any connected component of $\{\nabla_y {u^k}\neq 0\}$. 
This concludes the proof.
\end{proof}

\medskip 

There are some cases in which the directions $\omega_1,\ldots,\omega_n$ 
may be related or may coincide. In fact, as a corollary of Theorem \ref{teoprinc}, 
we prove that this happens under some additional assumptions on the functions $F_{kj}$. 
For this, we denote by $\Im(u^1,\ldots,u^n)$ the image of the map 
$(u^1,\ldots,u^n):\bbR^N\to\bbR^n$, that is
$$ \Im(u^1,\ldots,u^n):=\left\lbrace (u^1(X),\ldots,u^n(X)), X\in\bbR^N\right\rbrace. $$ 
Then, the following symmetry result holds: 
\begin{cor}\label{cor:stab}
Under the assumptions of Theorem \ref{teoprinc}, 
we assume that, for every $x\in\bbR^m$ and for every $j,k=1,\ldots,n$, $j\neq k$, 
\begin{equation}\begin{split}\label{intervalli}
&\mbox{there exist open intervals $I_1^x, \ldots, I_n^x\subset\bbR$ such that 
$\left(I_1^x\times\ldots\times I_n^x\right)\cap\Im(u^1,\ldots,u^n)\neq\varnothing$}\\ 
&\mbox{and $F_{kj}(x, \overline{u^1},\ldots,\overline{u^n})>0$ 
(or $F_{kj}(x, \overline{u^1},\ldots,\overline{u^n})<0$)}\\ 
&\mbox{for any 
$(\overline{u^1},\ldots,\overline{u^n})\in I_1^x\times\ldots\times I_n^x$.}
\end{split}\end{equation}  

Then there exist $C_{jk}:\R^N\longrightarrow\R$ and $D_{jk}:\R^m\longrightarrow \{-1,1\}$ such that
\begin{align}\label{cvb}
	\nabla_y u^j(X)=C_{jk}(X) \nabla_y u^k(X)\quad \mbox{and}\quad \omega_{j}(x)=D_{jk}(x) \omega_{k}(x).
\end{align}

If, in addition, $\mathcal{I}:=\bigcap_{i=1}^n\{\nabla_y u^i\neq 0\}\neq\varnothing$ is connected, then $\omega_1\equiv\ldots\equiv\omega_n$ in $\mathcal{I}$.
\end{cor}
\begin{proof}
By Theorem \ref{teoprinc}, for each $j,k=1,\ldots, n$,
\begin{align}\label{str2}
	u^j(x,y)=\bar u^j(x, \left\langle\omega_{j}(x), y\right\rangle),\quad \quad u^k(x,y)=\bar u^k(x, \left\langle \omega_{k}(x), y\right\rangle),
\end{align}
for some $\bar u^j,\bar u^k:\R^m\times \R\longrightarrow \R$ and $\omega_j,\omega_k:\R^m\longrightarrow \mathbb{S}^{N-m-1}$. 

Now, for any fixed $x\in\bbR^m$, arguing as in the proof of the second part of Theorem 
$1.8$ in \cite{DP}, one can prove that there exists a non-empty open set 
$V\subset\bbR^{N-m}$ such that $u^i(x,y)\in I_i^x$ and 
$\nabla_y u^i(x,y)\neq 0$, for all $y\in V$ and $i=1,\ldots,n$. 

Therefore, using (\ref{oss5}) and (\ref{intervalli}), we obtain that 
there exists $y_*\in V$ such that  
\begin{align}\label{strut}
sgn(F_{kj})\left\langle \nabla_y u^j(x,y_*), \nabla_y u^k(x,y_*)\right\rangle-|\nabla_y u^k(x,y_*)||\nabla_y u^j(x,y_*)|=0, 
\end{align}
for any $j,k=1,\ldots, n$, $j\neq k$. 
Moreover, from (\ref{str2}), we have that $\nabla_y u^j(x,y_*)$ is proportional to $\omega_j(x)$ and 
$\nabla_y u^k(x,y_*)$ is proportional to $\omega_k(x)$. 
Hence, (\ref{strut}) together with (\ref{str2}) implies (\ref{cvb}). 

If $\mathcal{I}\neq\varnothing$ is connected, we have that 
$$ \omega_i(x)=\omega_i, \qquad i=1,\ldots,n, $$
because, from Theorem \ref{teoprinc}, we know that 
each $\omega_i$ is constant in any connected component of 
$\left\lbrace \nabla_y u^i\neq 0\right\rbrace$. 

Now, plugging the functions in (\ref{str2}) into (\ref{strut}), we have that  
\begin{align*}
|\partial_z \overline{u}^j||\partial_z \overline{u}^k|\Big(\pm\left\langle \omega_j,\omega_k\right\rangle-1\Big)=0, \qquad j,k=1,\ldots, n, j\neq k,
\end{align*}
where $\partial_z \overline{u}^i$ denotes the derivative of the function 
$\overline{u}^i$ with respect to the last variable.  
Therefore, from the last equality, we deduce that, for every $j,k=1,\ldots, n$, 
\begin{align*}
\left\langle \omega_j,\omega_k\right\rangle=\pm 1, \qquad \mbox{in}\ \mathcal{I}.
\end{align*}
If $\omega_k=-\omega_j$, we have that 
$u^j(x,y)=\overline{u}^j(x,\left\langle\omega_j, y\right\rangle)$ and 
$u^k(x,y)=\overline{u}^k(x,\left\langle-\omega_j, y\right\rangle)$; then, we can define 
$\tilde{u}^k(x,y):=\overline{u}^k(x,-y)$, and obtain 
$u^k(x,y)=\tilde{u}^k(x,\left\langle\omega_j, y\right\rangle)$. 
This means that we can take $\omega_j=\omega_k$ up to renaming the function that 
describes $u^k$. 
Hence, we have that $\omega_j=\omega_k$ for every $j,k=1,\ldots,n$, 
and this concludes the proof. 
\end{proof}

\section{Monotone solutions and proof of Theorem \ref{teoprinc2}}
In this section we prove our symmetry result for monotone solutions to~\eqref{sistema}. 
First of all, we show that $F-$monotonicity implies stability 
(see Definitions~\ref{mon} and~\ref{defstab}).
\begin{prop}\label{implica}
If $(u^1,\ldots, u^n)$ is a $F-$monotone solution of (\ref{sistema}), then it is also stable.
\end{prop}
\begin{proof}
Choosing $\psi^i:=\frac{\xi_i^2}{u^i_{y_{N-m}}}\in W^{1,2}_0(\Omega)$ in (\ref{sistema4}), 
where $\xi_i\in C^{\infty}_c(\Omega)$ 
(we explicitly observe that this is possible thanks to Remark \ref{density}), we obtain
\begin{align}\label{dfg}
\sum_{j=1}^n\int_{\Omega}F_{ij}\, \frac{u_{y_{N-m}}^j}{u^i_{y_{N-m}}}\, \xi_i^2 & =\int_{\Omega}\left\langle \mathcal{A}^i\nabla u^i_{y_{N-m}}, \frac{(\nabla\xi_i^2)\, u^i_{y_{N-m}}-\xi_i^2\, \nabla u^i_{y_{N-m}}}{(u^i_{y_{N-m}})^2}\right\rangle \\
\nonumber
&=-\int_{\Omega}\frac{\xi_i^2}{(u^i_{y_{N-m}})^2}\left\langle \mathcal{A}^i\nabla u^i_{y_{N-m}},\nabla u^i_{y_{N-m}}\right\rangle \\ \nonumber
&\qquad +2\int_{\Omega}\frac{\xi_i}{u^i_{y_{N-m}}}\left\langle \mathcal{A}^i\nabla u^i_{y_{N-m}},\nabla\xi_i\right\rangle\\
\nonumber
&\leq \int_{\Omega}\left\langle \mathcal{A}^i\nabla\xi_i,\nabla\xi_i\right\rangle,
\end{align}
where in the last inequality we have used the fact that since $\mathcal{A}^i$ is positive definite then the following inequality holds for each $a\in\bbR$, $v,w\in\bbR^N$, $i\in\{1,\ldots, n\}$:
	\[2a\left\langle \mathcal{A}^iv,w\right\rangle-a^2\left\langle \mathcal{A}^iw,w\right\rangle-\left\langle \mathcal{A}^iv,v\right\rangle=\left\langle \mathcal{A}^i(v-aw),aw-v\right\rangle\leq 0.
\]
Summing over $i\in \{1,\ldots, n\}$ in (\ref{dfg}), we get
\begin{align}\label{qvb}
\sum_{i,j=1}^n\int_{\Omega}F_{ij}\, \frac{u^j_{y_{N-m}}}{u^i_{y_{N-m}}}\, \xi_i^2\leq \sum_{i=1}^n\int_{\Omega}\left\langle \mathcal{A}^i\nabla\xi_i,\nabla\xi_i\right\rangle, \qquad \forall \xi_i\in C^{\infty}_c(\Omega).
\end{align}
Moreover, 
\begin{align}\label{qvb2}
\sum_{i,j=1}^n F_{ij}\, \frac{u^j_{y_{N-m}}}{u^i_{y_{N-m}}}\, \xi_i^2&=\sum_{i,j=1}^n F_{ij}\, u^i_{y_{N-m}}u^j_{y_{N-m}}\, \frac{\xi_i^2}{(u^i_{y_{N-m}})^2} \\ \nonumber &= \sum_{i=1}^n F_{ii}\, \xi_i^2 + \sum_{i<j}F_{ij}\, u^i_{y_{N-m}}\, u^j_{y_{N-m}}\left(\frac{\xi_i^2}{(u^i_{y_{N-m}})^2}+\frac{\xi_j^2}{(u^j_{y_{N-m}})^2}\right)\\
\nonumber
&\geq \sum_{i,j=1}^nF_{ij}\, \xi_i\, \xi_j,
\end{align}
where in the last inequality we have used the fact that $F_{ij}\, u^i_{y_{N-m}}u^j_{y_{N-m}}\geq 0$ if $i<j$. Putting together (\ref{qvb}) and (\ref{qvb2}) we get the thesis.
\end{proof}

\begin{proof}[\textbf{Proof of Theorem \ref{teoprinc2}:}] 
By Proposition \ref{implica}, every monotone solution of (\ref{sistema}) is also stable. 
Moreover, the assumption in (\ref{assumdermiste}) is verified 
(it is enough to take $\theta^i:=u^i_{y_{N-m}}$, $\theta^j:=u^j_{y_{N-m}}$ which belong to $C^1(\bbR^N)$ thanks to \eqref{regassump}). 
Then, the hypotheses of Theorem \ref{teoprinc} are satisfied, and therefore 
we conclude that there exist $\bar u^i:\R^m\times\R\longrightarrow \R$ and $\omega_i\in \bbS^{N-m-1}$ such that 
\begin{align}\label{oss12}
		u^i(X)=u^i(x,y)=\bar u^i(x, \left\langle\omega_i, y\right\rangle)
\end{align}
for any $(x,y)\in\R^m\times \R^{N-m}$, $i\in\{1,\ldots, n\}$. 

Let us now assume that there exists $U\subseteq\bbR^m\times\bbR^n$ open such that, for every $j,k=1,\ldots, n$, $F_{jk}>0$ (or $F_{jk}<0$) in $U$. 
Using (\ref{oss5}) and (\ref{oss12}), we get  
\begin{align*}
\pm |\partial_z\overline{u}^j||\partial_z \overline{u}^k|\left\langle \omega_j,\omega_k\right\rangle=|\partial_z\overline{u}^j||\partial_z \overline{u}^k|\qquad \mbox{in}\ U,
\end{align*}
which implies $\left\langle \omega_j,\omega_k\right\rangle=\pm 1$, and hence $\omega_j=\omega_k=\omega$ (see the comments at the end of the proof of Corollary \ref{cor:stab}). This concludes the proof.
\end{proof}

\appendix
\section{}
In this appendix we analyze the assumptions made in Section $1$ 
in order to get our symmetry results. 

\subsection{Optimality of the assumptions}
We start observing that the regularity assumptions (\ref{regassump}) are fulfilled in a lot of interesting cases. Precisely, let $(u^1,\ldots, u^n)$  be a solution of (\ref{sistema}) with $u^i\in W^{1,p_i}(\Omega)\cap L^{\infty}(\Omega)$, and define 
$$ b^j_i(x,\nabla u^i(X)):=a_i(x,|\nabla u^i|(X))\partial_j u^i(X), \qquad G_i(X):=F_i(x,u^1(X),\ldots, u^n(X)).$$
Let us assume that for every $i=1,\ldots, n$
\begin{align}
\label{111}&b^j_i\in  C^{0}(\bbR^m\times \bbR^N)\cap C^1(\bbR^m\times \bbR^N\setminus\{0\}), \qquad j=1,\ldots,N\\
&\sum_{j,k=1}^N \frac{\partial b^j_i}{\partial\eta_k}(x,\eta)\xi_j\xi_k\geq \sigma(k+|\eta|)^{p_i-1}|\xi|^2,\\
&\sum_{j,k=1}^N\left|\frac{\partial b^j_i}{\partial\eta_k}(x,\eta)\right|\leq \Gamma (k+|\eta|)^{p_i-2},\\
&\sum_{j,k=1}^N\left|\frac{\partial b^j_i}{\partial x_k}(x,\eta)\right|\leq  \Gamma (k+|\eta|)^{p_i-2}|\eta|,\\
\label{100}&|G_i(X)|\leq \Gamma,
\end{align}
for all $\eta\in \bbR^N\setminus\{0\}$, $\xi\in \bbR^N$, $X\in\bbR^N$, 
with $p_i\geq 2$, $k\in [0,1]$, $\Gamma,\sigma>0$. 

Then, by \cite{Di,Le,Tolk,S},  we conclude that $u^i\in C^1(\bbR^N)\cap C^2(\{\nabla u^i\neq 0\})$ for each $i=1,\ldots, n$. 

Moreover, using  $(2.2.2)$ in \cite{Tolk} and Theorem $1.1$ and Proposition $2.2$ in \cite{DS}, we conclude that also the assumption $u^i\in W^{1,2}_{loc}(\bbR^N)$ is always verified if either $\{\nabla u^i= 0\}=\varnothing$ for $i=1,\ldots, n$ or $1<p<3$. 

Therefore, the functions~$u^i$ satisfy the regularity assumptions 
in~\eqref{regassump} provided the conditions in~\eqref{111}-\eqref{100} hold.

It is interesting to note that, as in the scalar case, 
the assumption $|\nabla u|\in L^{\infty}(\Omega)$ cannot be removed. 
Indeed, without such an assumption, one can find a solution which is not 
one-dimensional, according to the following proposition 
(see Proposition $3.1$ in~\cite{FSV}):
\begin{prop}
Let $k>0$ and $\psi\in C^1((k,+\infty))$ satisfying $\dot{\psi}(t)>0$ in $(k,+\infty)$ and $\lim_{t\to+\infty}\psi(t)=+\infty$. Then, there exists $a\in C^1((0,+\infty))$ strictly positive, and $u\in C^2(\bbR^N)$ which is a stable solution of 
	\[-\dive\big(a(|\nabla u(X)|)\nabla u(X)\big)=N
\]
and such that $|\nabla u(X)|=\psi(|X|)$ for any $|X|$ suitably large. 

Moreover, $u$ does not possess one-dimensional symmetry.
\end{prop}
We also mention that, proceeding exactly as in \cite{FSV,dipierro}, the assumption on the regularity of $F$, i.e. for any $(x,\xi^1,\ldots, \xi^{i-1},\xi^{i+1},\ldots, \xi^n)\in \R^m\times \R^{n-1}$ 
the map $\xi^i\to F(x,\xi^1,\ldots,\xi^i,\ldots, \xi^n)$ belongs to $C^2(\R)$, 
can be weakened requiring that the map $\xi^i\to F(x,\xi^1,\ldots,\xi^i,\ldots, \xi^n)$ is only $C_{loc}^{1,1}(\R)$. Notice that the extension to locally Lipschitz nonlinearities could be very interesting from a physical viewpoint; indeed, very often, physical applications are run by locally Lipschitz forces.

\subsection{On the $F$-monotonicity condition}
Proceeding in our discussion about the consistency of assumptions made in Section $1$, it is worth noticing that, as pointed out in \cite{FG}, the notion of $F-$monotonicity (see Definition \ref{mon}) seems to be crucial in order prove that a solution is one-dimensional. Indeed, let us consider the following system 
\begin{align}\label{sistwell}
-\Delta u+\nabla F(u)=0\quad \mbox{in}\ \R^2,
\end{align}
where $F:\R^2\longrightarrow \R$ is defined by:
\[
F(x_1,x_2):=(x_1-1)^2x_2^2+(x_2^2-1)^2.
\] 
Then, $F$ does not satisfy condition $ii)$ in the Definition \ref{mon}, indeed 
\[
F_{12}(x_1,x_2)=4x_2(x_1-1).
\]
Moreover,
\begin{align*}
&F\in C^2(\bbR^2),\qquad F((1,1))=0, \qquad F((1,-1))=0, \quad F(\xi)>0 \quad\mbox{for } \xi\neq (1,1),(1,-1),\\
& \nabla^2F((1,1))\geq  I, \qquad \nabla^2F((1,-1))\geq  I\\
&\nabla F(\xi)\cdot \xi\geq 0\quad \mbox{for}\ |\xi|\geq R_0,\quad \mbox{for some}\ R_0>1,
\end{align*}
which, by \cite[Theorem $1.1$]{ABG}, imply that there exist entire solutions $(u^1,u^2)$ of (\ref{sistwell}) which are not one-dimensional.

\subsection{Minimizers and stable solutions}
We point out some conditions which ensure the validity of \eqref{assump1}. As mentioned in the introduction, the system in~\eqref{sistema} 
is associated to a suitable energy functional. Precisely, let us define
\[
\lambda^i_1(x,t):=\frac{\partial a_i}{\partial t}(x,t)t+a_i(x,t), \qquad \lambda^i_2(x,t):=a_i(x,t), \qquad i=1,\ldots, n,
\]
and
\[
\Lambda^i_2(x,t):=\int_0^t \lambda^i_2(x,|\tau|)\tau\, \ud\tau. 
\]
Then, it a is a matter of computations that the energy functional 
related to~\eqref{sistema} is 
\begin{align}\label{defint}
I_{\Omega}(u^1,\ldots, u^n):=\sum_{i=1}^n\int_{\Omega} \Lambda_2^i(x,|\nabla u^i|)-F(x,u^1,\ldots, u^n).
\end{align}

According to~\cite{CValp}, we give the following definition: 
\begin{defi}
A family $(u^1,\ldots, u^n)$ is said to be a local minimizer for $I_{\Omega}$ if, for any bounded open set $U\subset\Omega$, $I_{U}(u^1,\ldots, u^n)$ is well-defined and finite, and
\begin{align*}
I_{U}(u^1+\psi^1,\ldots, u^n+\psi^n)\geq I_{U}(u^1,\ldots, u^n)
\end{align*}
for any $(\psi^1,\ldots, \psi^n)\in C^{\infty}_c(U,\bbR^n)$.
\end{defi}

The following lemma is the exact counterpart for systems of the result 
proved in \cite[Lemma B.1]{CValp} for the case of one equation.
\begin{lem}
Let $\Omega\subset\bbR^N$ be an open set. If $(u^1,\ldots, u^n)$ is a local minimizer of $I_{\Omega}$, then $(u^1,\ldots, u^n)$ is a weak solution of \eqref{sistema} and satisfies \eqref{stab}. 
\end{lem}
\begin{proof}
We start proving that every local minimizer $u=(u^1,\ldots, u^n)$ of $I_{\Omega}$ is a weak solution of \eqref{sistema}. To this end, let $U\subset \Omega$ be open and bounded and consider $\psi\in C^{\infty}_c(U)$. Then,
for every $i=1,\ldots, n$, we get
\begin{align*}
	\frac{\ud }{\ud s}\Big|_{s=0}I_U (u^1,\ldots,u^i+s\psi,\ldots, u^n)=0,
\end{align*}
and, recalling the definition of~$I_{U}$ in~\eqref{defint}, 
\begin{align*}
		\int_{\Omega}a_i(x,|\nabla u^i|(X))\left\langle \nabla u^i(X),\nabla\psi(X)\right\rangle \ud X=\int_{\Omega} F_{i}(x,u^1,\ldots, u^n)\, \psi(X)\, \ud X,
\end{align*}
which is the first part of the thesis. 

Finally, for every $\psi=(\psi^1,\ldots, \psi^n)\in C^{\infty}_c(U,\bbR^n)$,
\begin{equation*}\begin{split}
0&\leq \frac{\ud^2 }{\ud s^2}\Big|_{s=0}I_U (u^1+s\psi^1,\ldots, u^n+s\psi^n)\\
&=\sum_{i=1}^n\frac{\ud}{\ud s}\Big|_{s=0} \int_{\Omega}\Big(a_i(x,|\nabla u^i+s\nabla\psi^i|(X))\left\langle \nabla u^i+s\nabla \psi^i, \nabla\psi^i\right\rangle \\
&\qquad\qquad -F_{i}(x,u^1+s\psi^1,\ldots, u^n+s\psi^n)\, \psi^i(X)\Big)\ud X\\
& =\sum_{i=1}^n\int_{\Omega}\left\langle \mathcal{A}^i(x,\nabla u^i(X))\nabla \psi^i(X), \nabla \psi^i(X)\right\rangle\ud X \\
&\qquad -\sum_{i,j=1}^n\int_{\Omega} F_{ij}(x,u^1,\ldots, u^n)\, \psi^i(X)\, \psi^j(X)\, \ud X,
\end{split}\end{equation*}
and the proof is accomplished. 
\end{proof}

In the following proposition we give a sufficient condition for the 
assumption~\eqref{assump1} to hold for local minimizers.
\begin{prop}\label{port}
Let $N\leq 3$, and assume that, for each $i=1,\ldots,n$, there exists $C_i>0$ such that
\begin{align}
\label{sign}
&\lambda_1^i(x,t)>0, \qquad  \forall  x\in\bbR^m, t\in (0,+\infty),\\
\label{stimasop}
&\lambda_1^i(x,t)t^2,\lambda_2^i(x,t)t^2\leq C_i \Lambda_2^i(x,t), \qquad \forall  x\in\bbR^m, t\in (0,C_i].
\end{align}
Moreover,  we  assume that for all $x\in\bbR^m$, and $s,t\in [0,+\infty)$
\begin{align}
\label{zero}
&\Lambda_2^i(x,s)\geq 0, \\
\label{sub}
&\Lambda_2^i(x,s+t)\leq \bar C_i\Big[\Lambda_2^i(x,s)+\Lambda_2^i(x,t)\Big], \\
\label{cre}
&\Lambda_2^i(x,s)\leq \alpha_i(x)g_i(s),
\end{align}
for some $\bar{C}_i>0$, $\alpha_i\in L^{\infty}_{loc}(\bbR^m)$ and $g_i:[0,+\infty)\longrightarrow \bbR$ monotone increasing.
Finally, for all $x\in\bbR^m$ and $\xi\in\bbR^n$, we suppose that the following holds
\begin{align}\label{segno}
F(x,\xi)\leq 0,
\end{align} 
\begin{align}\label{cond}
F(x,\xi)=0, \qquad  \forall x\in\bbR^m,\  \forall \xi\in\bbS^{n-1}, 
\end{align}
and 
\begin{align}\label{lim}
\sup_{x\in\bbR^m, |\xi|\leq 1}|F(x,\xi)|\leq +\infty.
\end{align}

Then, assumption~\eqref{assump1} is satisfied for every local minimizer $(u^1,\ldots, u^n)$ of $I_{\bbR^N}$ such that 
$|u^i|+|\nabla u^i|\leq M$, $i=1,\ldots,n$.
\end{prop}
\begin{proof}
We start observing that, for each $i=1,\ldots, n$, and $x\in\bbR^m$, 
if $\xi,v\in\bbR^N$ with $|v|\leq 1$ and $|\xi|\leq M$, then
\begin{align}\label{disuglam}
|\xi|^2\left\langle \mathcal{A}^i(x,\xi)v,v\right\rangle \leq C_M \Lambda_2^i(x,|\xi|).
\end{align}
Indeed, by a simple calculation we get
\begin{align*}
|\xi|^2\left\langle \mathcal{A}^i(x,\xi)v,v\right\rangle&=\frac{\partial a_i}{\partial t}(x,|\xi|)|\xi|\left\langle \xi,v\right\rangle^2+a_i(x,|\xi|)|v|^2|\xi|^2\\
\nonumber
&=\lambda_1^i(x,|\xi|)\left\langle \xi, v\right\rangle^2+\lambda_2^i(x,|\xi|)\Big[|v|^2|\xi|^2-\left\langle \xi, v\right\rangle^2\Big]\\
\nonumber
&\leq\Big(\lambda_1^i(x,|\xi|)+\lambda_2^i(x,|\xi|)\Big)|v|^2|\xi|^2\\
\nonumber
&\leq  C_M \Lambda_2^i(x,|\xi|),
\end{align*}
where in the last inequality we have used \eqref{stimasop} and the fact that $|v|\leq 1$.

Let $R>1$ and take $\psi=(\psi^1,\ldots,\psi^n)\in (C^{\infty}_c(\bbR^N))^n$ with the property that, for each $i=1,\ldots, n$, $\psi^i=-1$ in $B_{R-1}$, $\psi^i=1$ on $\partial B_R$ and $|\nabla\psi^i|\leq M$ in $B_R\setminus B_{R-1}$. Let us define $$v^i(X):=\min\{u^i(X),\psi^i(X)\}, \qquad i=1,\ldots, n,$$
and observe that, by the minimality of $u$ and \eqref{segno},
\begin{align*}
\sum_{i=1}^n\int_{B_R}\Lambda_2^i(x,|\nabla u^i|)\leq I_{B_R}(u^1,\ldots, u^n)\leq I_{B_R}(v^1,\ldots, v^n). 
\end{align*}
By \eqref{zero}, \eqref{cond},\eqref{sub} and \eqref{lim} we get
\begin{align*}
\sum_{i=1}^n\int_{B_R}\Lambda_2^i(x,|\nabla u^i|)&\leq \int_{B_R\setminus B_{R-1}}\sum_{i=1}^n\Lambda_2^i(x,|\nabla v^i|)-F(x,v^1,\ldots, v^n)\\
&\leq \max\{\bar C_i\} \int_{B_R\setminus B_{R-1}}\sum_{i=1}^n\Big(\Lambda_2^i(x,|\nabla u^i|)+\Lambda_2^i(x,|\nabla \psi^i|)\Big)+\sup_{\bbR^m\times Q} |F|,
\end{align*}
where $Q:=[-1,1]\times \cdots \times[-1,1]$ is the cube in $\bbR^n$. 
Using \eqref{cre}, we obtain
\begin{align}\label{disugR}
\sum_{i=1}^n\int_{B_R}\Lambda_2^i(x,|\nabla u^i|)&\leq \max\{\bar C_i\} \int_{B_R\setminus B_{R-1}} \sum_{i=1}^n\alpha_i(x)\Big(g_i(|\nabla u^i|)+g_i(|\nabla \psi^i|)\Big)+\sup_{\bbR^m\times Q} |F|\\
	\nonumber
&\leq 2\max\{\bar C_i\}\int_{B_R\setminus B_{R-1}} \sum_{i=1}^n g_i(M)\sup_x \alpha_i +\sup_{\bbR^m\times Q} |F|\\
	\nonumber
&\leq CR^{N-1},
\end{align}
for some $C>0$. Finally, thanks to \eqref{disuglam}, \eqref{disugR} and the fact that $\mathcal{A}^i$ is positive definite, we have
\begin{align*}
\int_{B_R}|\nabla u^i|^2\left\langle \mathcal{A}^i(x,\nabla u^i)v,v\right\rangle\leq CR^{N-1},
\end{align*}
and, taking as $v$ the normalized eigenvector corresponding to $\overline{\mathcal{A}}^i$, 
we get the thesis.
\end{proof}

We conclude this Appendix providing an example of functional which satisfies the hypotheses in 
Proposition~\ref{port}, and hence the assumption~\eqref{assump1}, 
obtaining, from Theorem~\ref{teoprinc}, the one-dimensional symmetry for 
local minimizers. 
\begin{cor}
Let $N\leq3$, and let $\alpha\in L^{\infty}(\bbR^m)$ be strictly positive, 
and $F\in C^{2}(\bbR^2)$ such that $G:=\alpha F$ satisfies \eqref{segno}-\eqref{lim}. 
Suppose that $F_{12}$ does not change sign. 

For every $p_1,p_2\in (1,3)$, let us define the functional
\begin{align*}
	I_{\bbR^N}:=\int_{\bbR^N}|\nabla u^1(X)|^{p_1}+|\nabla u^2(X)|^{p_2}-\alpha(x)F(u^1(X),u^2(X))\ud X.
\end{align*}

Then, for every local minimizer $(u^1,u^2)$ of $I_{\bbR^N}$ such that $u^1\in W^{1,p_1}(\bbR^N)\cap L^{\infty}(\bbR^N)$, $u^2\in W^{1,p_2}(\bbR^N)\cap L^{\infty}(\bbR^N)$,  with $|\nabla u^1|,|\nabla u^2|\in L^{\infty}(\bbR^N)$, 
there exist $u^1_0,u^2_0:\bbR^m\times \bbR\longrightarrow\bbR$ 
and $\omega_1,\omega_2:\bbR^m\to\bbS^{N-m-1}$ such that 
	\[u^1(x,y)=u^1_0(x,\left\langle\omega_1(x), y\right\rangle),\qquad u^2(x,y)=u^2_0(x,\left\langle\omega_2(x), y\right\rangle),
\]
for any $(x,y)\in\bbR^m\times\bbR^{N-m}$. Moreover, $\omega_i$, $i=1,2$, is constant in any connected component of 
$\left\lbrace \nabla_y u^i\neq 0\right\rbrace$.
\end{cor}
\begin{proof}
The proof easily follows from Theorem \ref{teoprinc} and Proposition \ref{port} . Indeed,
$$\Lambda_2^1(x,t)=|t|^{p_1}, \qquad \Lambda_2^2(x,t)=|t|^{p_2}$$ 
satisfy conditions \eqref{zero}, \eqref{sub}, \eqref{cre} and 
\begin{equation*}\begin{split}
&\lambda_1^1(x,t)=(p_1-1)|t|^{p_1-2}, \qquad \lambda_1^2(x,t)=(p_2-1)|t|^{p_2-2}, \\ &\lambda_2^1(x,t)=|t|^{p_1-2}, \qquad \lambda_2^2(x,t)=|t|^{p_2-2}
\end{split}\end{equation*} 
satisfy \eqref{sign} and \eqref{stimasop}. Moreover, as proved in \cite{CVal}, both $\mathcal{A}^1,\mathcal{A}^2$ are positive definite for every $p_1,p_2>1$ and satisfy \eqref{Alim} when $p_1,p_2\geq 2$, and even for $p_1,p_2>1$ as long as $\{\nabla u^1=0\}=\{\nabla u^2=0\}=\varnothing$.
\end{proof}

\section*{Acknowledgments} The authors would like to thank \emph{Enrico Valdinoci} 
for his very deep and useful comments and suggestions.

\end{document}